\documentclass[11pt]{amsart}

\usepackage{amssymb}

\usepackage{enumitem}

\usepackage{graphicx}

\usepackage[margin=1in]{geometry}

\makeatletter
\@namedef{subjclassname@2020}{%
  \textup{2020} Mathematics Subject Classification}
\makeatother

\usepackage[T1]{fontenc}

\usepackage{thmtools}
\usepackage{thm-restate}
\usepackage{mathrsfs}
\usepackage{dsfont}
\usepackage{soul}

\newtheorem{theorem}{Theorem}[section]
\newtheorem{lemma}[theorem]{Lemma}
\newtheorem{proposition}[theorem]{Proposition}

\newtheorem{question}[theorem]{Question}
\newtheorem{corollary}[theorem]{Corollary}

\theoremstyle{definition}
\newtheorem{definition}[theorem]{Definition}

\theoremstyle{remark}
\newtheorem{remark}[theorem]{Remark}

\numberwithin{equation}{section}

\newcommand{\frakc}{\mathfrak{c}}

\newcommand{\eps}{\varepsilon}

\newcommand{\N}{\mathbb{N}}
\newcommand{\R}{\mathbb{R}}

\newcommand{\G}{\mathbb{G}}

\newcommand{\aA}{\mathcal{A}}
\newcommand{\bB}{\mathcal{B}}
\newcommand{\cC}{\mathcal{C}}

\newcommand{\uU}{\mathcal{U}}

\newcommand{\zZ}{\mathcal{Z}}

\newcommand{\concat}{{^\smallfrown}}

\DeclareMathOperator{\der}{\mathrm{d}}

\newcommand{\ol}{\overline}

\newcommand{\rstr}{\restriction}

\newcommand{\sm}{\setminus}
\newcommand{\sub}{\subseteq}
\DeclareMathOperator{\supp}{supp}

\newcommand{\seq}[2]{\big\langle#1\colon\ #2\big\rangle}
\newcommand{\seqn}[1]{\big\langle#1\colon\ n\io\big\rangle}

\newcommand{\seqk}[1]{\big\langle#1\colon\ k\io\big\rangle}

\newcommand{\seqi}[1]{\big\langle#1\colon\ i\io\big\rangle}

\newcommand{\finsub}[1]{\left[#1\right]^{<\omega}}

\newcommand{\io}{\in\omega}

\newcommand{\bo}{{\beta\omega}}

\newcommand{\fso}{\finsub{\omega}}

\newcommand{\Cantor}{2^\omega}

\newcommand{\noproof}{\hfill$\Box$}

\begin{document}


\baselineskip=17pt


\title[Grothendieck $C(K)$-spaces and the J.--N. theorem]{Grothendieck $C(K)$-spaces and the Josefson--Nissenzweig theorem}
\author[J. K\k{a}kol]{Jerzy K\k{a}kol}
\address{Faculty of Mathematics and Computer Science, Adam Mickiewicz University, ul. Uniwersytetu Pozna\'{n}skiego 4, 61-614 Pozna\'{n}, Poland, and Institute of Mathematics, Czech Academy of Sciences, \v{Z}itn\'{a} 25, 115 67 Praha 1, Czech Republic.}
\email{kakol@amu.edu.pl}
\author[D.\ Sobota]{Damian Sobota}
\address{Kurt G\"{o}del Research Center, Faculty of Mathematics, University of Vienna, Kolingasse 14--16, 1090 Wien, Austria.}
\email{ein.damian.sobota@gmail.com}
\author[L. Zdomskyy]{Lyubomyr Zdomskyy}
\address{Institute of Discrete Mathematics and Geometry, TU Wien, Wiedner Hauptstra\ss e 8--10/104, 1040 Wien, Austria}
\email{lzdomsky@gmail.com}

\date{}

\begin{abstract}
For a compact space $K$, the Banach space $C(K)$ is said to have \textit{the $\ell_1$-Grothendieck property} if every weak* convergent sequence $\big\langle\mu_n\colon\ n\in\omega\big\rangle$ of functionals on $C(K)$ such that $\mu_n\in\ell_1(K)$ for every $n\in\omega$, is weakly convergent. Thus, the $\ell_1$-Grothendieck property is a weakening of the standard Grothendieck property for Banach spaces of continuous functions. We observe that $C(K)$ has the $\ell_1$-Grothendieck property if and only if there does not exist any sequence of functionals $\big\langle\mu_n\colon\ n\in\omega\big\rangle$ on $C(K)$, with $\mu_n\in\ell_1(K)$ for every $n\in\omega$, satisfying the conclusion of the classical Josefson--Nissenzweig theorem. We construct an example of a separable compact space $K$ such that $C(K)$ has the $\ell_1$-Grothendieck property but it does not have the Grothendieck property. We also show that for many classical consistent examples of Efimov spaces $K$ their Banach spaces $C(K)$ do not have the $\ell_1$-Grothendieck property.
\end{abstract}


\maketitle

\section{Introduction}

A Banach space $E$ is \textit{a Grothendieck space} (or has \textit{the Grothendieck property}) if every weak* convergent sequence in the dual space $E^*$ is weakly convergent. If, for an infinite compact space $K$, the Banach space $C(K)$ of continuous real-valued functions on $K$ endowed with the supremum norm is Grothendieck, then let us also say that $K$ has \textit{the Grothendieck property}.

Many Banach spaces have been recognized as Grothendieck spaces, e.g. the space $\ell_\infty$ or more generally spaces $C(K)$ for extremely disconnected compact spaces $K$ (\cite{Gro53}), the space $H^\infty$ of bounded analytic functions on the unit disc (\cite{Bou83}), reflexive spaces, von Neumann algebras (\cite{Pfi94}), injective spaces, spaces $C(K)$ not containing complemented copies of $c_0$ (\cite{Cem84}), etc. On the other hand, if a compact space $K$ contains a non-trivial convergent sequence, then the Banach space $C(K)$ does not have the Grothendieck property (see \cite{MS22} for more general results). A full internal topological characterization of those compact spaces $K$ for which the spaces $C(K)$ are Grothendieck is however unknown (cf. \cite[Section 3]{Die73}).

Recall that the Josefson--Nissenzweig theorem asserts that for every infinite-dimensional Banach space $E$ there exists a sequence of normalized functionals in the dual space $E^*$ which is weak* convergent to $0$. The theorem constitutes a useful tool in studying the Grothendieck property of Banach spaces, especially those of the form $C(K)$, see e.g. \cite{Khu78}, \cite{Cem84}, \cite{Fre84}, \cite{KSZproc}. Note that in the case of $C(K)$-spaces, by the virtue of the Riesz representation theorem, the Josefson--Nissenzweig theorem simply states that every infinite compact space carries a sequence of normalized signed Radon measures which converges weak* to $0$ (see Section \ref{section:jn_ck} for the short proof of the theorem in this special case). Let us thus introduce the following convenient definition, where for a compact space $K$ by $\Delta(K)$ and $\ell_1(K)$ we mean the linear spaces of all Radon measures on $K$ with, respectively, finite supports and countable supports (cf. Section \ref{section:notation}).

\begin{definition}\label{def:jn_sequence}
A sequence $\seqn{\mu_n}$ of (signed) Radon measures on a compact space $K$ is \textit{a Josefson--Nissenzweig sequence} (in short, \textit{a JN-sequence}) on $K$ if it is weak* convergent to $0$ and $\big\|\mu_n\big\|=1$ for every $n\io$. If, in addition, $\mu_n\in\Delta(K)$ for each $n\io$, then we say that $\seqn{\mu_n}$ is \textit{a finitely supported Josefson--Nissenzweig sequence} (in short, \textit{an fsJN-sequence}). Similarly, if $\mu_n\in\ell_1(K)$ for each $n\io$, then $\seqn{\mu_n}$ is \textit{a countably supported Josefson--Nissenzweig sequence} (in short, \textit{a csJN-sequence}).
\end{definition}

Thus, the Josefson--Nissenzweig theorem asserts that all infinite compact spaces admit JN-sequences. On the other hand, not every infinite compact space admits an fsJN-sequence, e.g. it is proved in \cite{BKS19} that $\bo$, the \v{C}ech--Stone compactification of $\omega$, is such a space. In the first half of the paper we will study the relation between the Grothendieck property and the finitely supported Josefson--Nissenzweig property and generalize the latter result on $\bo$ by proving that no compact space $K$ with the Grothendieck property admits an fsJN-sequence or even a csJN-sequence.

We will need the following auxiliary definitions.

\begin{definition}\label{def:fsjnp}
A compact space $K$ has \textit{the finitely supported Josefson--Nissenzweig property} (in short, \textit{the fsJNP}) if $K$ admits an fsJN-sequence. Similarly, $K$ has \textit{the countably supported Josefson--Nissenzweig property} (in short, \textit{the csJNP)} if $K$ admits a csJN-sequence.
\end{definition}

\begin{definition}\label{def:ell1_gr}
For a compact space $K$, the Banach space $C(K)$ has \textit{the $\ell_1$-Grothendieck property}  (resp.
\textit{the $\Delta$-Grothendieck property}) if every weak*
convergent sequence of measures $\seqn{\mu_n\in\ell_1(K)}$ (resp.
$\seqn{\mu_n\in\Delta(K)}$) is weakly convergent. An infinite compact space $K$ has \textit{the $\ell_1$-Grothendieck property}  (resp.
\textit{the $\Delta$-Grothendieck property}) if the Banach space $C(K)$ has this property.
\end{definition}

Of course, if a compact space has the Grothendieck property, then it has the $\ell_1$-Grothendieck property, which further implies that it has the $\Delta$-Grothendieck property. By a routine computation and appealing to the Schur property of the space $\ell_1(K)$, it appears that the latter two properties are in fact the same as well as that they are equivalent to the lack of the finitely supported Josefson--Nissenzweig property.

\begin{restatable*}{theorem}{thmgrothfsjnp}\label{theorem:ell1_grothendieck_equiv_no_fsjnp}
Let $K$ be an infinite compact space. Then, the following are equivalent:
\begin{enumerate}
    \item $K$ has the $\ell_1$-Grothendieck property,
    \item $K$ has the $\Delta$-Grothendieck property,
    \item $K$ does not have the fsJNP,
    \item $K$ does not have the csJNP.
\end{enumerate}
\end{restatable*}

In the view of the aforementioned lack of any internal characterization of compact spaces with the Grothendieck property, it seems reasonable to ask whether the $\ell_1$-Grothendieck property is actually equivalent to the full Grothendieck property and hence whether the latter is equivalent to the lack of the fsJNP, too. It appears however that there exist compact spaces without the Grothendieck property but with the $\ell_1$-Grothendieck property.

\begin{restatable*}{theorem}{thmextrdiscgroth}\label{theorem:extr_disc_grothendieck}
For every infinite basically disconnected compact space $K$ there exists a compact space $L$ and a continuous surjection $\varphi\colon K\to L$ such that $L$ does not have the Grothendieck property but it has the $\ell_1$-Grothendieck property.
\end{restatable*}

\noindent (For the definition of a basically disconnected space, see Section \ref{section:notation}; recall here only that, e.g., every extremely disconnected space is basically disconnected.)

Setting $K=\bo$, we immediately obtain the following corollary.

\begin{restatable*}{corollary}{cornotgrellgr}\label{cor:notgr_ell1gr}
There exists a separable compact space $L$ such that it does not have the Grothendieck property but it has the $\ell_1$-Grothendieck property. \noproof
\end{restatable*}

The latter result complements Plebanek's construction of a (non-separable) compact space $K$ such that its every separable closed subspace has the Grothendieck property, yet $K$ itself does not have the property (\cite{Ple05}). (Note however that if every separable subspace of a given Banach space $E$ is contained in a Grothendieck subspace, then $E$ itself must be Grothendieck.)

\medskip

The second half of the paper is devoted to the study of the finitely supported Josefson--Nissenzweig property (or the $\ell_1$-Grothendieck property) of compact spaces representable as limits of inverse systems consisting of totally disconnected compact spaces. We are in particular interested in those systems in which all consecutive extensions are \textit{simple} (see Definition \ref{def:inv_sys_simple_ext}). We obtain the following general result.

\begin{restatable*}{theorem}{thminversefsjnp}\label{theorem:inverse_fsjnp}
Let $K$ be an infinite compact space. If $K$ is the limit of an inverse system of simple extensions, then $K$ has the fsJNP.
\end{restatable*}

Theorem \ref{theorem:inverse_fsjnp} has an immediate application to so-called Efimov spaces. Recall that \textit{the Efimov problem} is a long-standing open question asking whether every infinite compact space contains a non-trivial convergent sequence or a homeomorphic copy of $\bo$. An infinite compact space which contains neither of the mentioned subspaces is called \textit{an Efimov space}. Many examples of Efimov spaces have been obtained, but so far each of the constructions has made use of some additional set-theoretic assumption (see \cite[Section 6.2]{SZ19} for examples). Several of those examples are constructed as limits of inverse systems consisting of simple extensions, e.g. those of Fedorchuk (\cite{Fed76}, under the assumption of Jensen's Diamond Principle $\diamondsuit$), Dow and Pichardo-Mendoza (\cite{DPM09}, under the Continuum Hypothesis), or Dow and Shelah (\cite{DS13}, under Martin's axiom). Theorem \ref{theorem:inverse_fsjnp} implies that all of these examples have the fsJNP (Corollary \ref{cor:efimov_fsjnp}). 

We generalize Theorem \ref{theorem:inverse_fsjnp} to more complicated inverse systems, though of length at most $\frakc$ (see Definition \ref{def:tau_simple_extensions} and Theorem \ref{theorem:tau_simple_extensions_fsjnp}). As a result, based on constructions from \cite{dlV04}, \cite{dlV05}, and \cite{Back18}, we obtain that consistently there exist hereditarily separable Efimov spaces with various homogeneity properties satisfying the finitely supported Josefson--Nissenzweig property.

\begin{restatable*}{corollary}{cordelavegafsjnp}\label{cor:delavega_fsjnp}
Assume $\diamondsuit$.
\begin{enumerate}
    \item There exists a hereditarily separable totally disconnected rigid Efimov space satisfying the fsJNP.
    \item There exists a hereditarily separable totally disconnected Efimov space $K$ satisfying the fsJNP and such that any two non-empty clopen subsets of $F$ are homeomorphic.
\end{enumerate}
\end{restatable*}

\begin{restatable*}{corollary}{corbackefsjnp}\label{cor:backe_fsjnp}
Under $\diamondsuit$ there exists a hereditarily separable totally disconnected Efimov space $K$ satisfying the fsJNP and such that there are no disjoint infinite closed homeomorphic subspaces of $K$.
\end{restatable*}

Our results have some functional-analytic consequences. Namely, Cembranos \cite[Corollary 2]{Cem84} proved that for a compact space $K$ the space $C(K)$ is Grothendieck if and only if $C(K)$ does not contain any complemented copies of $c_0$, the Banach space of all sequences converging to $0$ endowed with the supremum norm. Banakh, K\k{a}kol, and \'Sliwa \cite[Theorem 1]{BKS19} proved that $K$ has the fsJNP if and only if the space $C_p(K)$ of continuous real-valued functions on $K$ endowed with the pointwise topology contains a complemented copy of the space $(c_0)_p$, that is, the space $c_0$ endowed with the product topology (inherited from $\mathbb{R}^\omega$). 
Consequently, by Theorem \ref{theorem:ell1_grothendieck_equiv_no_fsjnp} and Corollary \ref{cor:notgr_ell1gr}, we get the following two results.

\begin{restatable*}{corollary}{corellgrothcp}\label{cor:ell1_groth_c0p}
For a compact space $K$, the space $C(K)$ has the $\ell_1$-Grothendieck property if and only if the space $C_p(K)$ does not contain any complemented copies of $(c_0)_p$.
\end{restatable*}

\begin{restatable*}{corollary}{corexistscnotcp}\label{cor:exists_c0_not_c0p}
There exists a separable compact space $K$ such that $C(K)$ contains a complemented copy of $c_0$ but $C_p(K)$ does not contain any complemented copies of $(c_0)_p$.
\end{restatable*}

Note here that, by the Closed Graph Theorem, if a given space $C_p(K)$ contains a complemented copy of the space $(c_0)_p$, then $C(K)$ contains a complemented copy of $c_0$.

%

\section{Preliminaries and notation\label{section:notation}}

%
%
We use the following standard notions and symbols. If $X$ is a set and $A$ its subset, then $A^c=X\sm A$ and $\chi_A$ denotes the characteristic function of $A$ in $X$. The family of all subsets of $X$ is denoted by $\wp(X)$. The cardinality of a set $X$ is denoted by $|X|$. $\omega$ denotes the first infinite cardinal number and $\omega_1$ denotes the first uncountable cardinal number. The continuum, i.e. the size of the real line $\R$, is denoted either by $\frakc$ or $2^\omega$.

Throughout the paper we assume that all compact spaces are Hausdorff. The weight of a topological space $X$ is denoted by $w(X)$. If $X$ is a topological space and $A$ its subspace, then $\ol{A}^X$ denotes the closure of $A$ in $X$. We will often omit the superscript and write simply $\ol{A}$. $A^\circ$ and $\partial A$ denote the interior and the boundary of $A$ in $X$, respectively. $\beta X$ denotes the \v{C}ech--Stone compactification of $X$. The Cantor space will be usually denoted by $\Cantor$. We also usually identify $\omega$ with the space $\N$ of all natural numbers endowed with the discrete topology.

A subset $A$ of a topological space $X$ is \textit{a zero set} if there is a continuous function $f\colon X\to\R$ such that $A=f^{-1}(0)$, and it is \textit{a cozero set} if its complement $X\sm A$ is a zero set in $X$. Recall that a topological space $X$ is \textit{basically disconnected} if every cozero subset of $X$ has open closure, and it is \textit{extremely disconnected} if every open subset of $X$ has open closure. A compact space $K$ is \textit{totally disconnected} if it has a base consisting of clopen subsets. Of course, every extremely disconnected topological space is basically disconnected and every basically disconnected compact space is totally disconnected. Also, a compact space $K$ is extremely disconnected if and only if its Boolean algebra of clopen subsets is complete, and $K$ is basically disconnected if and only if this Boolean algebra is $\sigma$-complete.
%
%
%
%
%
%

If we say that \textit{$\mu$ is a measure on a compact space $K$}, then we mean that $\mu$ is a signed finite $\sigma$-additive measure defined on the Borel $\sigma$-field of $K$ and that $\mu$ is \textit{Radon}, i.e. both the positive and negative parts of $\mu$ are (outer and inner) regular and locally finite. We define \textit{the norm} $\|\mu\|$ of $\mu$ as
\[\|\mu\|=\sup\big\{|\mu(A)|+|\mu(B)|\colon\ A,B\sub X\text{ are Borel and disjoint}\big\}.\]
Since $K$ is compact, $\|\mu\|<\infty$. $|\mu|$ denotes \textit{the variation} of $\mu$---it follows that $|\mu|(K)=\|\mu\|$. Recall that due to the Riesz representation theorem the Banach space of all measures on $K$, endowed with the above norm, is isometrically isomorphic to the dual Banach space $C(K)^*$. For a function $f\in C(K)$ we simply set $\mu(f)=\int_Kfd\mu$.

A measure $\mu$ on a compact space $K$ is \textit{a probability measure} if $\mu(A)\ge0$ for every Borel $A$ and $\|\mu\|=1$. We say that $\mu$ \textit{vanishes at points} (or, is \textit{non-atomic}) if $\mu(\{x\})=0$ for every $x\in K$.

If $\mu$ is a measure on a compact space $K$, then by $\supp(\mu)$ we denote \textit{the support} of $\mu$, i.e. the smallest closed subset $L$ of $K$ such that for every open subset $U\sub K\sm L$ we have $|\mu|(U)=0$. We will say that $\mu$ is \textit{finitely (countably) supported} if $\supp(\mu)$ is a finite (countable) set. For a subset $S\sub K$, the spaces of all finitely supported measures on $K$ and all countably supported measures on $K$, whose supports are contained in $S$, are denoted by $\Delta(S)$ and $\ell_1(S)$, respectively. 
If $x\in K$, then by $\delta_x$ we mean \textit{the point measure} (or \textit{the Dirac measure}) concentrated at $x$ and defined as $\delta_x(A)=\chi_A(x)$, for every Borel $A$. Each element $\mu$ of $\ell_1(K)$ may be written as the sum  $\mu=\sum_{x\in\supp(\mu)}\alpha_x\cdot\delta_x$, 
for some non-zero $\alpha_x\in\R$ and every $x\in\supp(\mu)$. Similarly, the variation of $\mu$ may be written as $|\mu|=\sum_{x\in\supp(\mu)}\big|\alpha_x\big|\cdot\delta_x$, and thus the norm $\|\mu\|$ is equal to $\sum_{x\in\supp(\mu)}\big|\alpha_x\big|$. 
%
%
%

Recall that, for a compact space $K$, $\ell_1(K)$ is a complemented linear subspace of $C(K)^*$ and that it has \textit{the Schur property}, i.e. every weakly convergent sequence in $\ell_1(K)$ is also norm convergent. Note also that a sequence $\seqn{\mu_n}$ of measures on $K$ is \textit{weak* convergent} to a measure $\mu$ if $\lim_{n\to\infty}\mu_n(f)=\mu(f)$ for every $f\in C(K)$ (again by the virtue of the Riesz representation theorem), and that it is \textit{weakly convergent} to $\mu$ if $\lim_{n\to\infty}\mu_n(B)=\mu(B)$ for every Borel subset $B$ of $K$ (cf. \cite[Theorem 11, page 90]{Die84}).
%

\section{The Josefson--Nissenzweig theorem for $C(K)$-spaces\label{section:jn_ck}}

Josefson \cite{Jos75} and Nissenzweig \cite{Nis75} proved their theorem for general infinite-dimensional Banach spaces and both of the proofs are rather long, technical, and intricate. However, when we restrict our attention only to the Banach spaces of continuous functions on infinite compact spaces, then it appears that the theorem may be proved in a much easier way. Below we present one of such proofs, relying on measure-theoretic tools such as the Maharam theorem\footnote{Let us note here that another basic proof for the case of $C(K)$-spaces can be also easily extracted from the proof of the general Josefson--Nissenzweig theorem due to Behrends \cite{Beh94,Beh95}; his proof exploits Rosenthal's $\ell_1$-lemma and Banach limits, however, since $\ell_1$ embeds into $C(K)^*$, we may omit the application of the $\ell_1$-lemma and go directly to Case 2 of the proof.}. We also provide a sketch of the argument proving that if a compact space admits an fsJN-sequence, then it admits one with pairwise disjoint supports.


\subsection{A short measure-theoretic proof of the Josefson--Nissenzweig theorem for $C(K)$-spaces\label{section:plebanek_proof}}

We will prove that every infinite compact space admits a JN-sequence. Let thus $K$ be an infinite compact space. If $K$ is a scattered space, i.e. every subset of $K$ contains an isolated point in the inherited topology, then it is a simple folklore fact that $K$ contains a non-trivial sequence $\seqn{x_n}$ convergent to some point $x\in K$. A sequence $\seqn{\mu_n}$ of measures defined for each $n\io$ by the formula $\mu_n=\frac{1}{2}\big(\delta_{x_n}-\delta_x\big)$ is then a JN-sequence on $K$.

If $K$ is not scattered, then the proof requires more work. By \cite[Theorem 19.7.6]{Sem71}, there is a non-atomic probability measure $\mu$ on $K$. It follows from the celebrated Maharam theorem (\cite{Mah42}, see also \cite{Fre89}) that there exists a sequence $\seqn{B_n}$ of $\mu$-independent Borel subsets of $K$ such that $\mu\big(B_n\big)=1/2$ for every $n\io$. (The $\mu$-independence of $\seqn{B_n}$ means here that for every finite sequence $n_1,\ldots,n_k$ of distinct natural numbers and every sequence $\eps_1,\ldots,\eps_k\in\{-1,1\}$ we have:
\[\mu\Big(\bigcap_{i=1}^kB_{n_i}^{\eps_i}\Big)=\prod_{i=1}^k\mu\big(B_{n_i}^{\eps_i}\big)=1/2^k,\]
where $A^1=A$ and $A^{-1}=K\sm A$ for a subset $A$ of $K$.) For each $n\io$ define the measure $\mu_n$ as follows:
\[\mu_n(A)=\mu\big(B_n\cap A\big)-\mu\big(B_n^c\cap A\big),\]
where $A$ is a Borel subset of $K$; then, $\big\|\mu_n\big\|=1$. The sequence $\seqn{\mu_n}$ is a desired JN-sequence on $K$. Indeed, note that
\[\mu_n(g)=\int_Kg\cdot\big(\chi_{B_n}-\chi_{B_n^c}\big){\der}\mu\]
for every $n\io$ and $g\in L_1(\mu)$.\footnote{The symbols $L_1(\mu)$ and $L_\infty(\mu)$ denote the usual Banach spaces associated with the measure space $(K,Bor(K),\mu)$.} By the $\mu$-independence of the sequence $\seqn{B_n}$ and the generalized Riemann--Lebesgue lemma (\cite[Page 3]{Tal84Pettis}), the bounded sequence $\seqn{\chi_{B_n}-\chi_{B_n^c}}$ of functions in $L_\infty(\mu)$ has the property that
\[\lim_{n\to\infty}\int_Kg\cdot\big(\chi_{B_n}-\chi_{B_n^c}\big){\der}\mu=0\]
for every $g\in L_1(\mu)$, which implies that $\lim_{n\to\infty}\mu_n(g)=0$ for every $g\in C(K)$, too. The proof is thus finished.

\subsection{Disjointly supported fsJN-sequences\label{section:dsfsjnseq}}

We now present a brief sketch of the proof of the theorem asserting that if a compact space admits an fsJN-sequence, then it carries one with pairwise disjoint supports---for a detailed proof of a more general result, which does not require Kadets--Pe\l czy\'{n}ski--Rosenthal's Subsequence Splitting Lemma, we refer the reader to \cite[Section 4]{MSZ}. 

\begin{theorem}\label{theorem:disjointly_supported}
Let $K$ be a compact space with the fsJNP. Then, there exists an fsJN-sequence $\seqn{\theta_n}$ which is disjointly supported, i.e. $\supp\big(\theta_n\big)\cap\supp\big(\theta_k\big)=\emptyset$ for every $n\neq k\io$.
\end{theorem}


\begin{proof}
Let $\seqn{\mu_n}$ be an fsJN-sequence on a compact space $K$. Put $S = \bigcup_{n\io}\supp\big(\mu_n\big)$ and set $\nu=\sum_{n\io}\big|\mu_n\big|/2^{n+1}$, so each $\mu_n$ is absolutely continuous with respect to the probability measure $\nu$ on $S$. By Kadec--Pe\l czy\'{n}ski--Rosenthal's Subsequence Splitting Lemma applied to the sequence of the Radon--Nikodym derivatives $\seqn{\!\der\!\mu_n/\!\der\!\nu}$ in the Banach space $L_1(S,\wp(S),\nu)$ (cf. \cite[Lemma 5.2.7 and Theorem 5.2.8]{AK06}), Eberlein--\v{S}mulian's theorem, and Schur's property of the Banach space $\big(\ell_1(S),\|\cdot\|_1\big)$, there exist a subsequence $\seqk{\mu_{n_k}}$, a sequence $\seqk{A_k}$ of pairwise disjoint finite subsets of $S$, and a measure $\mu$ of the form $\sum_{x\in S}\alpha_x \delta_x$, $\alpha_x\in\R$ ($x\in S$), such that the sequence $\seqk{\mu_{n_k}\rstr \big(S \setminus A_k\big)}$ converges in norm to $\mu$, that is, $\lim_{k\to\infty}\big\|\mu_{n_k}\rstr \big(S \setminus A_k\big)-\mu\big\|=0$. Of course, $\|\mu\|\le 1$.

For every $k\io$ we set:
\[\nu^1_k=\mu_{n_k}\rstr A_k\quad\text{and}\quad\nu^2_k=\mu_{n_k}\rstr\big(S \setminus A_k\big);\]
so, $\mu_{n_k} = \nu^1_k + \nu^2_k$. Obviously, $\seqk{\nu^1_k}$ is disjointly supported. Observe that for some $\varepsilon > 0$ we have $\big\|\nu^1_k\big\|> \varepsilon$ for almost all $k\in\omega$. Otherwise, we would have a subsequence $\seqi{\nu^1_{k_i}}$ converging in norm to the zero measure $0$. Then, $\seqi{\mu_{n_{k_i}}}$ would converge in norm to $\mu$, so $\|\mu\| = 1$. On the other hand, for every $f\in C(K)$, since $\lim_{i\to\infty}\mu_{n_{k_i}}(f) = 0$, it would hold
\[\mu(f)=\sum_{x\in S}f(x)\mu(\{x\})=0,\]
so the Riesz representation theorem would imply that $\mu=0$, 
a contradiction.

Since $\seqk{\mu_{n_k}}$ is an fsJN-sequence and $\seqk{\nu^2_k}$ converges in norm (and hence weak*) to $\mu$, $\seqk{\nu^1_k}$ is weak* convergent to $-\mu$. Let $\rho_k = \nu^1_{2k} - \nu^1_{2k+1}$ for $k\in \omega$. It follows that $\seqk{\rho_k}$ is weak* null. Clearly, the supports of $\rho_k$'s are pairwise disjoint. Since $\big\|\rho_k\big\| > 2\varepsilon$ for every $k\io$, the sequence $\seqk{\theta_k}$, defined for every $k\io$ by $\theta_k=\rho_k\big/\big\|\rho_k\big\|$, is the desired disjointly supported fsJN-sequence.
\end{proof}

\section{The $\ell_1$-Grothendieck property and the fsJNP\label{section:grothendieck}}

In this section we present the proof of the equivalence of the $\ell_1$-Grothendieck property and the lack of the finitely supported Josefson--Nissenzweig property. The argument basically relies on some routine calculations and the Schur property of spaces of the form $\ell_1(K)$.

\thmgrothfsjnp
\begin{proof}
Let us first show the equivalence (1)$\Leftrightarrow$(2). As $\Delta(K)\sub\ell_1(K)$, the $\ell_1$-Grothendieck property immediately implies the $\Delta$-property. Assume now that $K$ has the $\Delta$-Grothendieck property and let $\seqn{\mu_n\in\ell_1(K)}$ be weak* null. For each $n\io$ find a finite set $F_n\sub K$ such that $\big\|\mu_n\rstr\big(K\sm F_n\big)\big\|<1/n$. It follows that $\seqn{\mu_n\rstr F_n}$ is weak* null, too. Indeed, for every $f\in C(K)$ and $n\io$ we have:
\[\big|\big(\mu_n\rstr F_n\big)(f)\big|\le\big|\mu_n(f)\big|+\big|\big(\mu_n\rstr\big(K\sm F_n\big)\big)(f)\big|<\big|\mu_n(f)\big|+\|f\|_\infty\cdot1/n,\]
so $\lim_{n\to\infty}\big(\mu_n\rstr F_n\big)(f)=0$. 

Similarly, for every $x^{**}\in C(K)^{**}$  we have:
\[\big|x^{**}\big(\mu_n\big)\big|\le\big|x^{**}\big(\mu_n\rstr F_n\big)\big|+\big|x^{**}\big(\mu_n\rstr\big(K\sm F_n\big)\big)\big|\le\]
\[\big|x^{**}\big(\mu_n\rstr F_n\big)\big|+\big\|x^{**}\big\|\cdot1/n,\]
so $\lim_{n\to\infty}\big|x^{**}\big(\mu_n\big)\big|=0$, since by the $\Delta$-Grothendieck property the sequence $\seqn{\mu_n\rstr F_n}$ is weakly null. This proves that $K$ has the $\ell_1$-Grothendieck property.

\medskip

Next, we show the equivalence (3)$\Leftrightarrow$(4). If $K$ has the fsJNP, then $K$ has trivially also the csJNP, since $\Delta(K)\sub\ell_1(K)$. Let us thus assume that $K$ has the csJNP and let $\seqn{\mu_n}$ be a csJN-sequence. For each $n\io$ let $F_n$ be a finite subset of $\supp\big(\mu_n\big)$ such that $\big\|\mu_n\rstr\big(K\sm F_n\big)\big\|<1/n$, so $\big\|\mu_n\rstr F_n\big\|>1-1/n$. For every $n\io$ define the measure $\nu_n$ on $K$ as follows:
\[\nu_n=\big(\mu_n\rstr F_n\big)\Big/\big\|\mu_n\rstr F_n\big\|;\]
then, $\nu_n\in\Delta(K)$ and $\big\|\nu_n\big\|=1$. For every $f\in C(K)$ we have:
\[\big|\nu_n(f)\big|=\big|\big(\mu_n\rstr F_n\big)(f)\big|\Big/\big\|\mu_n\rstr F_n\big\|\le\]
\[\Big(\big|\mu_n(f)\big|+\big|\big(\mu_n\rstr\big(K\sm F_n\big)\big)(f)\big|\Big)\Big/\big\|\mu_n\rstr F_n\big\|<\]
\[\Big(\big|\mu_n(f)\big|+\|f\|_\infty/n\Big)\Big/\big(1-1/n\big),\]
so $\lim_{n\to\infty}\nu_n(f)=0$, since $\lim_{n\to\infty}\mu_n(f)=0$, which implies that $\seqn{\nu_n}$ is weak* null. It follows that $\seqn{\nu_n}$ is an fsJN-sequence on $K$ and hence $K$ has the fsJNP.

\medskip

Finally, we show the equivalence (1)$\Leftrightarrow$(4). Assume that $K$ has the $\ell_1$-Grothendieck property. Suppose additionally that $K$ has the csJNP, so there is a csJN-sequence $\seqn{\mu_n}$ on $K$. Put $L=\bigcup_{n\io}\supp\big(\mu_n\big)$. By the $\ell_1$-Grothendieck property, $\seqn{\mu_n}$ is weakly null. The Hahn--Banach theorem implies then that $\seqn{\mu_n}$ is also weakly null as a sequence of elements of the space $\ell_1(L)$. Since the space $\ell_1(L)$ has the Schur property, it follows that
\[\lim_{n\to\infty}\big\|\mu_n\big\|_{C(K)^*}=\lim_{n\to\infty}\big\|\mu_n\big\|_{\ell_1(L)}=0,\]
which contradicts the fact that $\big\|\mu_n\big\|_{C(K)^*}=1$ for every $n\io$. Thus, $K$ cannot have the csJNP.


Assume now that $K$ does not have the csJNP. Suppose additionally that $K$ does not have the $\ell_1$-Grothendieck property, so there exists a weak* null sequence $\seqn{\mu_n\in\ell_1(K)}$ which is not weakly null. Since the weak topology is weaker than the norm topology, it follows that there exist a subsequence $\seqk{\mu_{n_k}}$ and $\eps>0$ such that for every $k\io$ we have $\big\|\mu_{n_k}\big\|>\eps$. For every $k\io$ define the measure $\nu_k=\mu_{n_k}\big/\big\|\mu_{n_k}\big\|$ and note that $\nu_k\in\ell_1(K)$ and $\big\|\nu_k\|=1$. Also, for every $f\in C(K)$ we have:
\[\big|\nu_k(f)\big|=\big|\mu_{n_k}(f)\big|\big/\big\|\mu_{n_k}\big\|<\big|\mu_n(f)\big|/\eps,\]
so $\lim_{k\to\infty}\nu_k(f)=0$. It follows that $\seqk{\nu_k}$ is a csJN-sequence on $K$, which is a contradiction.
\end{proof}

\begin{corollary}\label{cor:gr_no_fsJNP}
If a compact space $K$ has the Grothendieck property, then it does not have the fsJNP.\noproof
\end{corollary}

Using the same methods as in the proof of the implication (1)$\Rightarrow$(4), one can obtain the following fact (which also implies Corollary \ref{cor:gr_no_fsJNP}).

\begin{proposition}\label{prop:grothendieck_null_ctbl_part}
Let $K$ be a compact space with the Grothendieck property. Assume that $\seqn{\mu_n}$ is a JN-sequence on $K$. For each $n\io$ write $\mu_n=\nu_n+\theta_n$, where $\nu_n\in\ell_1(K)$ and $\theta_n$ is non-atomic. Then, $\big\|\nu_n\big\|\to0$, or equivalently $\big\|\theta_n\big\|\to1$, as $n\to\infty$.\noproof
\end{proposition}

By the discussion at the end of Introduction (in particular, by \cite[Theorem 1]{BKS19}), we also get the following $C_p$-theoretic characterization of the $\ell_1$-Grothendieck property (cf. \cite[Corollary 2]{Cem84}).

\corellgrothcp

\section{The $\ell_1$-Grothendieck property vs. Grothendieck property\label{section:ell1_gr_no_gr}}

In his unpublished note \cite{Ple05}, Plebanek constructed in ZFC a compact space $K$ such that its every separable closed subspace $L$ has the Grothendieck property, but $K$ itself does not have the property. It follows that $K$ is not separable, but it has the $\ell_1$-Grothendieck property.

Following the ideas of \cite{Ple05}, we will construct in
this section a separable compact space without the Grothendieck property, but with the $\ell_1$-Grothendieck
property.

\begin{lemma}\label{lemma:probability_grothendieck}
Let $K$ be a totally disconnected compact space, $\mu$ a
probability measure on $K$, and $\seqn{A_n}$ a sequence of
mutually disjoint clopen subsets of $K$ such that $\mu\big(A_n\big)>0$ for
every $n\io$. Define the set $F$ as follows: $x\in F$ if and only if
for every clopen neighborhood $U$ of $x$ the following inequality is
satisfied:
\[\limsup_{n\to\infty}\frac{\mu\big(A_n\cap U\big)}{\mu\big(A_n\big)}>0.\]
Then, $F$ is closed and non-empty, and the quotient space $K/F$ does
not have the Grothendieck property.
\end{lemma}
\begin{proof}
We first show that $F\neq\emptyset$. Assume for the sake of contradiction that for every $x\in K$ there exists its clopen neighborhood $U_x$ such that $\lim_n\mu\big(A_n\cap U_x\big)/\mu\big(A_n\big)=0$. By compactness of $K$, there exists a finite cover $U_{x_1},\ldots,U_{x_k}$ of $K$. We then have:
\[1=\lim_{n\to\infty}\frac{\mu\big(A_n\cap K\big)}{\mu\big(A_n\big)}\le\sum_{i=1}^k\lim_{n\to\infty}\frac{\mu\big(A_n\cap U_{x_i}\big)}{\mu\big(A_n\big)}=0,\]
a contradiction.

Let us now prove that $K/F$ does not have the Grothendieck property.
Let $\varphi\colon K\to K/F$ be the quotient map.  Denote
$p=\varphi[F]$. For every $n\io$ define a measure $\mu_n$ on $K/F$
as follows:
\[\mu_n(A)=\frac{\mu\big(A_n\cap\varphi^{-1}[A]\big)}{\mu\big(A_n\big)},\]
where $A$ is a clopen subset of $K/F$. Then, $\mu_n$ converges weak* to $\delta_p$ on $K/F$. Indeed, if $A$ is a clopen in $K/F$ not containing $p$, then $\varphi^{-1}[A]\cap F=\emptyset$ and hence, by compactness of $\varphi^{-1}[A]$ and the definition of $F$, we have $\limsup_n\mu\big(A_n\cap\varphi^{-1}[A]\big)/\mu\big(A_n\big)=0$, and so $\lim_n\mu_n(A)=0$. On the other hand, if $p\in A$, then:
\[\lim_{n\to\infty}\mu_n(A)=\lim_{n\to\infty}\Big(\mu_n(K/F)-\mu_n(A^c)\Big)=1-\lim_{n\to\infty}\mu_n(A^c)=1-0=1.\]
Had $K/F$ the Grothendieck property, $\mu_n$ would converge weakly
to $\delta_p$ and hence $\mu_n(\{p\})$ would converge to $1$, which
is not the case, since $A_n\cap F=\emptyset$ for every $n\io$ as
elements of $\seqn{A_n}$ are mutually disjoint.
\end{proof}

\begin{lemma}\label{lemma:z_pseudo_intersection}
Let $K$ be a basically disconnected compact space, $\mu$ a
probability measure on $K$, and $\seqn{A_n}$ a sequence of
mutually disjoint clopen subsets of $K$ such that $\mu\big(A_n\big)>0$ for
every $n\io$. Let $F$ be defined as in Lemma
\ref{lemma:probability_grothendieck}. Let $\zZ$ denote the family of
all clopen subsets $C$ of $K$ such that
$\lim_{n\to\infty}{\mu(A_n\cap C)}\big/{\mu(A_n)}=0,$ i.e.,
$$ \zZ=\big\{C\subset K\colon\  C\mbox{ is clopen and }C\cap F=\emptyset\big\}. $$
 Then $\zZ$  has the following pseudo-union-like
property: for every sequence $\seqn{C_n}$ of elements in $\zZ$ there exists
$C\in\zZ$ such that \[\forall n\in\omega\: \exists m\in\omega\:
(C_n\setminus C\subset\bigcup_{j\leq m}A_j).\]
\end{lemma}
\begin{proof}
The proof is similar to the standard one showing that the density ideal on $\omega$ is a P-ideal. Namely, inductively find a strictly increasing sequence $\seqk{n_k}$ of indices such that
\[\frac{\mu\Big(A_n\cap\bigcup_{i=0}^{k}C_i\Big)}{\mu\big(A_n\big)}<\frac{1}{k+1}\]
for every $n>n_k$ and $k\io$. Put:
\[C=\ol{\bigcup_{k\io}\Big(C_k\sm\bigcup_{j=0}^{n_k}A_j\Big)}.\]
Since $K$ is basically disconnected, $C$ is a clopen set. It
follows easily that for every $k\io$ we have:
\[C_k\sm C\sub\bigcup_{j=0}^{n_k}A_j.\]
We shall now show that $C\in\zZ$. Fix $n\io$, $n>n_0$, and let $k\io$ be such that $n_k<n\le n_{k+1}$. We have:
\[\frac{\mu\big(A_n\cap C\big)}{\mu\big(A_n\big)}\le\frac{\mu\Big(A_n\cap\bigcup_{i=0}^{k}C_i\Big)}{\mu\big(A_n\big)}+\frac{\mu\big(A_n\cap D\big)}{\mu\big(A_n\big)}<\frac{1}{k+1}+\frac{\mu\big(A_n\cap D\big)}{\mu\big(A_n\big)},\]
where:
\[D=C\sm\Big(\bigcup_{i\io}\big(C_i\sm\bigcup_{j=0}^{n_i}A_j\big)\Big),\]
so it is just the set added to
$\bigcup_{i\io}\big(C_i\sm\bigcup_{j=0}^{n_i}A_j\big)$ after taking
the closure. Since $k\to\infty$ if $n\to\infty$, it is enough to
show that $A_n\cap D=\emptyset$,  whence $\mu\big(A_n\cap
D\big)/\mu\big(A_n\big)=0$. Assume to the contrary that there exists
$x\in A_n\cap D$. Since $x\in D$, every neighborhood of $x$
intersects  $C_k\setminus\bigcup_{j=0}^{n_k}A_j$ for infinitely many
$k$. In particular, $A_n$ must intersect
$C_k\setminus\bigcup_{j=0}^{n_k}A_j$ for some $k$ with $n_k>n$,
which is impossible.
\end{proof}

\thmextrdiscgroth
\begin{proof}
Let $\mu$, $\seqn{A_n}$, $F$, and $\zZ$ be as in Lemmas
\ref{lemma:probability_grothendieck} and
\ref{lemma:z_pseudo_intersection}. Note that since $K$ is not scattered we may assume that $\mu$ vanishes on points (see \cite[Theorem 19.7.6]{Sem71}). Put
$L=K/F$ and let $\varphi$ be the quotient map. It follows from Lemma
\ref{lemma:probability_grothendieck} that $L$ does not have the
Grothendieck property. For the sake of contradiction assume that $L$
does not have the $\ell_1$-Grothendieck property either, so there is
an fsJN-sequence $\seqn{\mu_n}$ on $L$. By Theorem \ref{theorem:disjointly_supported}, 
we may assume that $\supp\big(\mu_n\big)\cap\supp\big(\mu_k\big)=\emptyset$ for $n\neq k\io$. We may also
assume that $\varphi[F]\cap\supp\big(\mu_n\big)=\emptyset$ for every
$n\io$, and hence  for every $n\io$ we can find $C_n\in\zZ$ such
that $\supp\big(\mu_n\big)\sub\varphi\big[C_n\big]$. Let $C\in\zZ$
be like in Lemma \ref{lemma:z_pseudo_intersection} for the sequence
$\seqn{C_n}$, i.e., $C_n\setminus C\subset\bigcup_{j=0}^{m_n}A_j$
for some increasing number sequence $\seqn{m_n}$. By
Proposition~\ref{prop:grothendieck_null_ctbl_part},
 the Grothendieck property of $K$ (and hence of
$\varphi[C]$) yields
\[\lim_{n\to\infty}\big\|\mu_n\rstr\varphi[C]\big\|=0,\]
which together with $\|\mu_n\rstr\varphi[C_n]\big\|=1$ and
$C_n\setminus C\subset\bigcup_{j=0}^{m_n}A_j$ gives
\[\lim_{n\to\infty}\big\|\mu_n\rstr\bigcup_{j=0}^{m_n}\varphi\big[A_j\big]\big\|=1.\]
On the other hand, since  for every $Q\in\fso$ we have:
\[\lim_{n\to\infty}\big\|\mu_n\rstr\bigcup_{j\in Q}\varphi\big[A_j\big]\big\|=0,\]
it follows that there exists a subsequence $\seqk{\mu_{n_k}}$ and a
sequence $\seqk{Q_k\in\fso}$ of pairwise disjoint sets such that for
every $k\io$ we have:
\[\big\|\mu_{n_k}\rstr\bigcup_{j\in Q_k}\varphi\big[A_j\big]\big\|>1/2.\]
Let $B_k$ be a clopen subset of $\bigcup_{j\in Q_k}A_j$ containing
$\supp\mu_{n_k}\cap\bigcup_{j\in Q_k}A_j$ and such that
$\frac{\mu(B_k\cap A_j)}{\mu(A_j)}<1/2^k$ for all $j\in Q_k$ (this
is the only place where we use that $\mu$ vanishes on points). Set
$D=\overline{\bigcup_{k\in\omega}B_k}$ and note that $D$ is a clopen
subset of $K$ such that $D\cap A_j=B_k\cap A_j$ for all $k\in\omega$
and $j\in Q_k$, and $D\cap A_j=\emptyset $ for
$j\in\omega\setminus\bigcup_{k\in\omega}Q_k$. Thus
\[\lim_{j\to\infty}\frac{\mu(A_j\cap D)}{\mu(A_j)}=0,\]
which means $D\in\zZ$, i.e., $D\cap F=\emptyset$. It follows from
the above that
\[\big\|\mu_{n_k}\rstr\varphi[D]\big\|=\big\|\mu_{n_k}\rstr\bigcup_{j\in Q_k}\varphi\big[A_j\big]\big\|>1/2\]
for every $k\io$, which is a contradiction, since $K$ (and hence
$\varphi[D]$) has the Grothendieck property.
\end{proof}

Considering $K=\bo$, we obtain the following corollary.

\cornotgrellgr

The latter result, together with Corollary \ref{cor:ell1_groth_c0p}, yield the following functional-analytic consequence.

\corexistscnotcp

We will also need the following definition.

\begin{definition}\label{def:ell1_gr_ba}
A Boolean algebra $\aA$ has \textit{the Grothendieck property} (resp. \textit{the $\ell_1$-Grothendieck property}) if its Stone space $St(\aA)$ has this property.
\end{definition}

Recall that every ($\sigma$-)complete Boolean algebra has the Grothendieck property. By the Stone duality we also obtain the following corollary saying that every ($\sigma$-)complete Boolean algebra may be slimmed down in such a way that it loses the Grothendieck property but still preserves the $\ell_1$-Grothendieck property.

\begin{corollary}\label{cor:grothendieck_subalgebras}
For every $\sigma$-complete Boolean algebra $\aA$ there exists a subalgebra $\bB\sub\aA$ such that $\bB$ does not have the Grothendieck property but it has the $\ell_1$-Grothendieck property. \noproof
\end{corollary}

Analyzing the proof of Theorem \ref{theorem:extr_disc_grothendieck}, it seems that the theorem might also hold for those infinite totally disconnected compact spaces which are the Stone spaces of Boolean algebras with e.g. Haydon's Subsequential Completeness Property (\cite{Hay81}). We do not know however how far the $\sigma$-completeness of $\aA$ can be weakened in Corollary \ref{cor:grothendieck_subalgebras}, which motivates the following question.

\begin{question}
Let $\aA$ be a Boolean algebra with the Grothendieck property. Does there exist a Boolean subalgebra $\bB$ of $\aA$ which fails to have the Grothendieck property but which nonetheless has the $\ell_1$-Grothendieck property?
\end{question}

\section{Systems of simple extensions and the fsJNP}

In this section we will show, combining several already known results, that the limit of every inverse system of simple extensions of compact spaces has the finitely supported Josefson--Nissenzweig property (Theorem \ref{theorem:inverse_fsjnp}). This yields a corollary that many consistent examples of Efimov spaces from the literature (e.g. \cite{Fed76}, \cite{DPM09}, \cite{DS13}), constructed under such axioms as the Continuum Hypothesis or Martin's axiom, have the fsJNP as well. In Subsection \ref{section:tau_simple_ext} we will generalize the results.

\subsection{Systems of simple extensions\label{section:simple_ext}}


We start this subsection with recalling what an inverse system of simple extensions is. For general information on limits of inverse systems, see \cite[Chapters 2.5 and 3.2]{Eng89}.

\begin{definition}\label{def:inv_sys_simple_ext}
An inverse system $\big\langle K_\alpha,\pi_\alpha^\beta\colon\alpha\le\beta\le\delta\big\rangle$ of totally disconnected compact spaces is \textit{a system of simple extensions} if
\begin{itemize}
    \item it is \textit{continuous}, i.e. for every limit ordinal $\gamma\le\delta$ the space $K_\gamma$ is the limit of the inverse system $\big\langle K_\alpha,\pi_\alpha^\beta\colon\alpha\le\beta\le\gamma\big\rangle$,
	\item $K_0$ is a singleton,
    \item for every $\alpha<\delta$ the space $K_{\alpha+1}$ is \textit{a simple extension of }$K_\alpha$, i.e. there is $x_\alpha\in K_\alpha$ such that $\Big|\big(\pi_\alpha^{\alpha+1}\big)^{-1}\big(x_\alpha\big)\Big|=2$ and for every $y\in K_\alpha\sm\big\{x_\alpha\big\}$ it holds $\Big|\big(\pi_\alpha^{\alpha+1}\big)^{-1}(y)\Big|=1$.
\end{itemize}
\end{definition}

To prove the main result of this section, Theorem \ref{theorem:inverse_fsjnp}, we also need to provide several definitions concerning complexity of probability measures on compact spaces.

\begin{definition}
\textit{The Maharam type} of a probability measure $\mu$ on a compact space $K$ is the minimal cardinality of a family $\cC$ of Borel subsets of $K$ such that for every Borel subset $B$ of $K$ and $\eps>0$ there exists $C\in\cC$ such that $\mu(B\triangle C)<\eps$.
\end{definition}

\noindent Equivalently, the Maharam type of a probability measure $\mu$ is the density of the Banach space $L_1(\mu)$ of all $\mu$-integrable functions. For more information on the topic, see \cite{Mah42}, \cite{Fre89}, or \cite{PS14}.

A notion closely related to the countable Maharam type is the uniform regularity, introduced by Babiker \cite{Bab77} and later intensively studied in \cite{Pol82}, \cite{Mer96}, and \cite{KP11}. 

\begin{definition}\label{def:unif_reg_measure}
A probability measure $\mu$ on a compact space $K$ is \textit{uniformly regular} if there exists a countable family $\cC$ of zero subsets of $K$ such that for every open subset $U$ of $K$ and every $\eps>0$ there exists $F\in\cC$ such that $F\sub U$ and $\mu(U\sm F)<\eps$.
\end{definition}

\noindent Uniformly regular measures are also  called
\textit{strongly countably determined} (cf. \cite{Pol82}). Note
that every uniformly regular probability measure $\mu$ has
necessarily separable support and countable Maharam type.

Recall that a subset $Y$ of a topological space $X$ is \textit{a $\G_\delta$-set} if there exists a countable collection $\uU$ of open subsets of $X$ such that $Y=\bigcap\uU$. An element $x\in X$ is called \textit{a $\G_\delta$-point} if $\{x\}$ is a   $\G_\delta$-set of $X$. Every zero subset of a topological space is a closed $\G_\delta$-set and one can easily show that in a normal space every closed $\G_\delta$-set is a zero set, thus the definition of uniformly regular measures may be stated in terms of closed $\G_\delta$-sets. 

\begin{proposition}\label{prop:unif_reg_g_delta}
Let $\mu$ be a uniformly regular measure on a compact space $K$ and $x\in K$ be such that $\mu(\{x\})>0$. Then, $x$ is a $\G_\delta$-point.
\end{proposition}
\begin{proof}
Let $\cC$ be a countable collection of zero sets (closed $\G_\delta$'s) witnessing that $\mu$ is uniformly regular. Put $\cC'=\big\{F\in\cC\colon\ x\in F\big\}$. It follows that $\{x\}=\bigcap\cC'$. To see this, assume that there is $y\in\bigcap\cC'$ such that $x\neq y$. Put $\eps=\mu(\{x\})$; so $\eps>0$. Using the regularity of $\mu$, it is easy to see that there is an open neighborhood $U$ of $x$ not containing $y$ and such that $\mu(U\sm\{x\})<\eps/3$. However, there is no $F\in\cC$ such that $F\sub U$ and $\mu(U\sm F)<\eps/3$, which contradicts the choice of $\cC$. Indeed, if there was such $F$, then from the fact that
\[\eps\le\mu(U)=\mu(U\sm\{x\})+\mu(\{x\})<4\eps/3\]
we would get that $\mu(F)>2\eps/3$, so necessarily $x\in F\in\cC'$ and hence $y\in F\sub U$, whereas $y\not\in U$, which would be a contradiction. Since the intersection of a countable collection of $\G_\delta$-sets is $\G_\delta$, $x$ is a $\G_\delta$-point.
\end{proof}

\begin{remark}\label{remark:atom_unif_reg_conv_seq}
Note that from Proposition \ref{prop:unif_reg_g_delta} it immediately follows that if for a uniformly regular measure $\mu$ on an infinite compact space $K$ there is a non-isolated point $x\in K$ such that $\mu(\{x\})>0$), then $K$ contains a non-trivial sequence converging to $x$.
\end{remark}

\begin{definition}
Let $\mu$ be a probability measure on a compact space $K$. We say that a sequence $\seqn{x_n}$ of points in $K$ is \textit{$\mu$-uniformly distributed} if the sequence of sums
$\seq{\frac{1}{n}\sum_{i=0}^{n-1}\delta_{x_i}}{n>0}$ converges weak* to
$\mu$. 
\end{definition}

Uniformly distributed sequences constitute a useful tool for
investigating various properties of probability measures as they
allow to treat those measures in a way similar to the classical
Jordan measure on the real line, see e.g. \cite{KN74}, \cite{Los78}, \cite{Los79}, or \cite{Mer96}.


Let us say that a sequence $\seqn{x_n}$ in a space $X$ is \textit{injective} if $x_n\neq x_{n'}$ for every $n\neq n'\io$. The following proposition will be crucial for the proof of the main theorem of this section.

\begin{proposition}\label{prop:unif_reg_meas_unif_distr_injective_seq}
If $\mu$ is a non-atomic uniformly regular measure on a compact space $K$, then $\mu$ admits a uniformly distributed injective sequence.
\end{proposition}
\begin{proof}
Let $\mu^\infty$ be the product Radon measure induced by $\mu$ on the product space $K^\omega$. \cite[Corollary 2.8]{Mer96} implies that $\mu^\infty(S)=1$, where $S$ denotes the subspace of $K^\omega$ consisting of all $\mu$-uniformly distributed sequences in $K$. Since by the non-atomicity of $\mu$ the subspace $T$ of $K^\omega$ consisting of all non-injective sequences in $K$ satisfies $\mu^\infty(T)=0$ (by, e.g., the Fubini--Tonelli theorem \cite[Theorem 7.27]{Fol99}), it follows that there exists a $\mu$-uniformly distributed sequence in $S$ which is injective.
\end{proof}

\begin{proposition}\label{prop:uds_fsjnp}
Let $\mu$ be a probability measure on a compact space $K$ and $\seqn{x_n}$ be a $\mu$-uniformly distributed injective sequence in $K$. Then, $K$ has the fsJNP.
\end{proposition}
\begin{proof}
Let $\omega=\bigcup_{n\io}P_n$ be a partition of $\omega$ into finite sets such that
\[\max P_n<\min P_{n+1}\quad\text{and}\quad\big|P_n\big|\big/\max P_n\ge1/2\]
for every $n\io$ (cf. the standard proof that $\sum_{n\io}1/n=\infty$). For every $n\io$ let us write:
\[\nu_n=\frac{1}{\max P_{n+1}}\sum_{k<\max P_{n+1}}\delta_{x_k}-\frac{1}{\max P_n}\sum_{k<\max P_n}\delta_{x_k}.\]
Then, by the injectivity of $\seqn{x_n}$,
\[\big\|\nu_n\big\|\ge\frac{1}{\max P_{n+1}}\cdot\big|P_{n+1}\big|\ge1/2.\]
Since either sum in the definition of $\nu_n$'s is weak* convergent to $\mu$, $\seqn{\nu_n}$ converges weak* to $0$. Normalizing $\mu_n=\nu_n/\big\|\nu_n\big\|$, $\seqn{\mu_n}$ is an fsJN-sequence on $K$.
\end{proof}

The following theorem was proved by Borodulin-Nadzieja \cite{PBN07} in the language of minimally generated Boolean algebras, the notion dual to limits of inverse systems of simple extensions (see \cite{Kop89}).

\begin{theorem}\label{theorem:pbn}
The following assertions hold for every totally disconnected compact space $K$:
\begin{enumerate}
    \item \cite[Theorem 4.6]{PBN07} $K$ carries either a uniformly regular measure or a measure of uncountable Maharam type;
    \item \cite[Theorem 4.9]{PBN07} If $K$ is the limit of an inverse system of simple extensions, then every measure on $K$ has countable Maharam type. In particular, there exists a uniformly regular measure on $K$.\noproof
\end{enumerate}
\end{theorem}

\noindent Let us recall here that D\v{z}amonja and Plebanek \cite[Lemma 4.1]{DP07} proved that if an inverse system of simple extensions has length at most $\omega_1$, then every measure on its limit is uniformly regular.

We are in the position to prove the main theorem of this section.

\thminversefsjnp
\begin{proof}
If $K$ is a scattered space, then it contains a non-trivial convergent sequence and hence it trivially has the fsJNP. So, let us assume that there exists an infinite closed subspace $L$ of $K$ containing no isolated points. By \cite[Proposition 2.5.6]{Eng89}, $L$ is also the limit of an inverse system of simple extensions. By Theorem \ref{theorem:pbn}.(2), there exists a uniformly regular
measure $\mu$ on $L$. If $\mu$ has an atom, then it is non-isolated in $L$, hence, by Proposition
\ref{prop:unif_reg_g_delta} and Remark
\ref{remark:atom_unif_reg_conv_seq}, $L$ contains a
non-trivial convergent sequence and so $L$ has the
fsJNP. If on the other hand $\mu$ is non-atomic, then
Proposition \ref{prop:unif_reg_meas_unif_distr_injective_seq} implies that $\mu$ admits a uniformly distributed injective sequence. Now, Proposition \ref{prop:uds_fsjnp} yields an fsJN-sequence on $L$. Since the fsJNP is inherited by superspaces, $K$ has the fsJNP, too.
\end{proof}

The following consequence of Theorems \ref{theorem:ell1_grothendieck_equiv_no_fsjnp} and \ref{theorem:inverse_fsjnp} generalizes the well-known fact that no minimally generated Boolean algebra has the Grothendieck property.

\begin{corollary}\label{cor:min_gen_no_ell_1_gr}
If $\aA$ is a minimally generated Boolean algebra, then $\aA$ does not have the $\ell_1$-Grothendieck property.\noproof
\end{corollary}

As a corollary to Theorem \ref{theorem:inverse_fsjnp} we also obtain that many Efimov spaces constructed in the literature have the fsJNP.

\begin{corollary}\label{cor:efimov_fsjnp}
If $K$ is an Efimov space obtained as the limit of an inverse system of simple extensions, then $K$ has the fsJNP. In particular, the examples of Efimov spaces by Fedorchuk (under $\diamondsuit$; see \cite{Fed76}), Dow and Pichardo-Mendoza (under the Continuum Hypothesis; see \cite{DPM09}), or Dow and Shelah (under Martin's axiom; see \cite{DS13}) have the fsJNP.\noproof
\end{corollary}

Let us note that consistently there exist Efimov spaces with the Grothendieck property and hence without the fsJNP, see e.g. \cite{Tal80}, \cite{Bre06}, or \cite{SZ19}.

In the next section we will prove that some other classes of Efimov spaces do have the fsJNP, too.

\subsection{$\tau$-simple extensions\label{section:tau_simple_ext}}

The aim of this section is to generalize Theorem \ref{theorem:inverse_fsjnp} to a broader class of inverse systems of compact spaces (however of length at most $\frakc$). We start with the following simple observations.

\begin{lemma}\label{lemma:simple_ext_boundaries}
Let $\big\langle K_\alpha,\pi_\alpha^\beta\colon\alpha\le\beta\le\delta\big\rangle$ be an inverse system of simple extensions. For every $\alpha<\delta$ and every subset $X\sub K_{\alpha+1}$ we have $\partial\pi_\alpha^{\alpha+1}[X]\sm\pi_\alpha^{\alpha+1}[\partial X]\sub\big\{x_\alpha\big\}$.
\end{lemma}
\begin{proof}
Fix $\alpha<\delta$ and a subset $X\sub K_{\alpha+1}$. For the sake of contradiction, assume there is $x\in\partial\pi_\alpha^{\alpha+1}[X]\sm\pi_\alpha^{\alpha+1}[\partial X]$ such that $x\neq x_\alpha$. Let $V$ be a clopen subset of $K_\alpha$ such that $x_\alpha\in V$ but $x\not\in V$. Put $X'=X\sm\big(\pi_\alpha^{\alpha+1}\big)^{-1}[V]$. Since $\big(\pi_\alpha^{\alpha+1}\big)^{-1}[V]$ is closed, $x\in\partial\pi_\alpha^{\alpha+1}[X']\sm\pi_\alpha^{\alpha+1}[\partial X']$. But, as $K_{\alpha+1}\sm\pi_\alpha^{\alpha+1}[V]$ is homeomorphic to $K_\alpha\sm V$, we have $\partial\pi_\alpha^{\alpha+1}[X']=\pi_\alpha^{\alpha+1}[\partial X']$, a contradiction.
\end{proof}

Recall that a continuous surjection $f\colon X\to Y$ between two topological spaces $X$ and $Y$ is \textit{irreducible} if $f[A]\neq Y$ for every closed proper subset $A$ of $X$.

\begin{lemma}\label{lemma:simple}
Let $\big\langle K_\alpha,\pi_\alpha^\beta\colon\alpha\le\beta\le\delta\big\rangle$ be an infinite inverse system of simple extensions such that for every $\alpha\ge\omega$ the space $K_\alpha$ has no isolated points. Then, $\pi^\beta_\alpha$ is irreducible for any $\omega\le\alpha<\beta\le\delta$.
\end{lemma}
\begin{proof}
By the continuity of this inverse system, it is enough to prove that $\pi^{\alpha+1}_\alpha$ is irreducible for every $\alpha<\delta$. But this is fairly simple, and is actually proved in the first paragraph of the proof of Proposition \ref{prop:delavega_omega_simple}---one just needs to consider the case when $\big|G_\alpha\big|=1$.
\end{proof}

Let us also note that if $\big\langle
K_\alpha,\pi_\alpha^\beta\colon\alpha\le\beta\le\delta\big\rangle$
is an inverse system of simple extensions, then
$w\big(K_\alpha\big)=w\big(K_{\alpha+1}\big)$ for every
$\alpha<\delta$. (Here, $w(\cdot)$ denotes the weight of a space.) Motivated by these three
observations, we introduce the following generalization of systems
of simple extensions.

\begin{definition}\label{def:tau_simple_extensions}
Let $\tau\le\frakc$ be a cardinal number. An inverse system $\big\langle K_\alpha,\pi_\alpha^\beta\colon\alpha\le\beta\le\delta\big\rangle$ of totally disconnected compact spaces is \textit{a system of $\tau$-simple extensions} if
\begin{itemize}
    \item it is continuous,
    \item $K_0=2^\omega$ and each $K_\alpha$ is perfect, i.e. has no isolated
    points,
    \item for every $\alpha<\delta$ the space $K_{\alpha+1}$ is \textit{a $\tau$-simple extension of }$K_\alpha$, i.e. $\Big|\partial\pi_\alpha^{\alpha+1}[U]\sm\pi_\alpha^{\alpha+1}[\partial U]\Big|\le\tau$ for every closed subset $U\sub K_{\alpha+1}$,
    \item the map $\pi_\alpha^\beta$ is irreducible for every $\alpha<\beta\le\delta$,
    \item $w\big(K_\alpha\big)=w\big(K_{\alpha+1}\big)$ for every $\alpha<\delta$.
\end{itemize}
\end{definition}

Theorem \ref{theorem:tau_simple_extensions_fsjnp} states that systems of $\tau$-simple extensions of length at most $\frakc$ have the fsJNP. In order to show this, we need first to prove several technical results.

\begin{lemma}\label{lemma:image_boundary}
Let $K$ and $L$ be two compact spaces and $f\colon K\to L$ a continuous surjection. Assume that for a clopen subset $U\sub K$ the interior $\big(f[U]\cap f[K\sm U]\big)^\circ=\emptyset$. Then, $f[U]\cap f[K\sm U]=\partial f[U]\cup\partial f[K\sm U]$.
\end{lemma}
\begin{proof}
We have:
\[\partial f[U]=\ol{f[U]}\cap\ol{L\sm f[U]}=f[U]\cap\ol{L\sm f[U]}\sub\]
\[f[U]\cap\ol{f[K\sm U]}=f[U]\cap f[K\sm U],\]
where the only inclusion
follows from the surjectivity of $f$ and the last
equality from the closedness of $f$. We show
similarly that $\partial f[K\sm U]\sub f[U]\cap f[K\sm U]$, whence
we get:
\[\partial f[U]\cup\partial f[K\sm U]\sub f[U]\cap f[K\sm U].\]

Let now $x\in f[U]\cap f[K\sm U]$. By the assumption, for every open neighborhood $V$ of $x$ we have $V\not\subseteq f[U]\cap f[K\sm U]$, so either $V\sm f[U]\neq\emptyset$ or $V\sm f[K\sm U]\neq\emptyset$. If for every $V$ we have $V\sm f[U]\neq\emptyset$, then $x\in\partial f[U]$. 
So let us assume that there exists an open neighborhood $V$ of $x$ such that $V\sm f[U]=\emptyset$, equivalently $V\sub\big(f[U]\big)^\circ$. It follows that $x\in\partial f[K\sm U]$, since otherwise there is  an open neighborhood $W$ of $x$ such that $W\sm f[K\sm U]=\emptyset$, so $W\sub\big(f[K\sm U]\big)^\circ$, and hence:
\[x\in V\cap W\sub f[U]^\circ\cap f[K\sm U]^\circ=\big(f[U]\cap f[K\sm U]\big)^\circ,\]
a contradiction. We get thus:
\[f[U]\cap f[K\sm U]\sub \partial f[U]\cup\partial f[K\sm U].\]
\end{proof}

The next result, Proposition \ref{prop:transport_fsjn_seq}, shows how fsJN-sequences can be recovered from the Cantor space via continuous surjections. Note that if for every $n\io$, $i\in 2$, and $s\in 2^n$ we put $x_s^i=s\concat(i)$, where $(i)$ denotes the constant sequence of length $\omega$ all of whose members equal $i$, then the measures defined as
\[\mu_n=\frac{1}{2^{n+1}}\sum_{s\in 2^n}\big(\delta_{x_s^1}-\delta_{x_s^0}\big)\]
form an fsJN-sequence on the Cantor space $\Cantor$. Recall that $\lambda$ denotes the standard product measure on $\Cantor$.

\begin{proposition}\label{prop:transport_fsjn_seq}
Let $Y$ be a totally disconnected compact space and  $f\colon Y\to\Cantor$ a continuous surjection such that $\lambda\big(f[U]\cap f[Y\setminus U]\big)=0$ for every clopen $U\sub Y$. For every $n\io$, $i\in 2$ and $s\in 2^n$ fix $y_s^i\in f^{-1}\big(x^i_s\big)$ and define the measure on $Y$ as follows:
\[\nu_n=\frac{1}{2^{n+1}}\sum_{s\in 2^n}\big(\delta_{y^1_s}-\delta_{y^0_s}\big).\]
Then, the sequence $\seqn{\nu_n}$ is an fsJN-sequence on $Y$.
\end{proposition}
\begin{proof}
We need only to show that $\seqn{\nu_n(U)}$ converges to $0$ for every clopen $U\sub Y$. Fix $\eps>0$. Since $\lambda\big(f[U]\cap f[Y\setminus U]\big)=0$, it follows that $f[U]\cap f[Y\setminus U]$ has empty interior in $\Cantor$. By Lemma \ref{lemma:image_boundary}, there is a clopen set $B\sub\Cantor$ such that $\lambda(B)<\eps$ and
\[\partial f[U]\cup\partial f[Y\sm U]=f[U]\cap f[Y\setminus U]\sub B.\]
It follows that $f[U]\setminus B$ and $f[Y\setminus U]\setminus B$ are also clopen sets. Since
\[f^{-1}\big[f[U]\setminus B\big]\sub U\]
and
\[f^{-1}\big[f[Y\setminus U]\setminus B\big]\sub Y\setminus U,\]
we have that $y^i_s\in U$ if $x^i_s\in f[U]\setminus B$, and $y^i_s\in Y\setminus U$ if $x^i_s\in f[Y\setminus U]\setminus B$. Let $n_0\io$ be such that $x^0_s\in f[U]\setminus B$ if and only if $x^1_s\in f[U]\setminus B$, for all $n\ge n_0$ and $s\in 2^n$. Then,
\[\big|\nu_n(U)\big|=\Big|\frac{1}{2^{n+1}}\sum_{s\in 2^n}\big(\delta_{y^1_s}(U)-\delta_{y^0_s}(U)\big)\Big|=\]
\[\frac{1}{2^{n+1}}\Big|\sum_{s\in 2^n}\big(\delta_{y^1_s}(U)-\delta_{y^0_s}(U)\big)-\sum_{s\in 2^n}\big(\delta_{x^1_s}(f[U]\setminus B)-\delta_{x^0_s}(f[U]\setminus B)\big)\Big|=\]
\[\frac{1}{2^{n+1}}\Big|\sum_{s\in 2^n}\big(\delta_{y^1_s}(U)-\delta_{x^1_s}(f[U]\setminus B)\big)-\sum_{s\in 2^n}\big(\delta_{y^0_s}(U)-\delta_{x^0_s}(f[U]\setminus B)\big)\Big|\le\]
\[\frac{1}{2^{n+1}}\Big(\big|\sum_{f(y^1_s)\in B}\delta_{y^1_s}(U)\big|+\big|\sum_{f(y^0_s)\in B}\delta_{y^0_s}(U)\big|\Big)=\]
\[\frac{1}{2^{n+1}}\Big(\big|\sum_{x^1_s\in B}\delta_{y^1_s}(U)\big|+\big|\sum_{x^0_s\in B}\delta_{y^0_s}(U)\big|\Big)\le\]
\[\frac{1}{2^{n+1}}\Big(\big|\big\{s\in 2^n\colon x^1_s\in B\big\}\big|+\big|\big\{s\in 2^n\colon x^0_s\in B\big\}\big|\Big)\]
for all $n\geq n_0$.
Let $n_1\io$ be such that for every $n\ge n_1$ there exists $S_n\sub 2^n$ for which $B=\bigcup_{s\in S_n}[s]$.
Then $\lambda(B)=\big|S_n\big|/2^n<\eps$ for all $n\ge n_1$. Then, for all $n\ge\max\{n_0,n_1\}$ we have:
\[\big|\nu_n(U)\big|\le\frac{1}{2^{n+1}}\Big(\big|\big\{s\in 2^n\colon\ x^1_s\in B\big\}\big|+\big|\big\{s\in 2^n\colon\ x^0_s\in B\big\}\big|\Big)\le\]
\[\frac{1}{2^{n+1}}\cdot 2\cdot\big|S_n\big|=\frac{\big|S_n\big|}{2^n}<\eps,\]
which completes the proof.
\end{proof}

\begin{lemma}\label{lemma:size_boundaries_induction}
Fix three cardinal numbers $\delta,\kappa,\tau<\frakc$, where $\kappa$ is infinite. Assume that $\big\langle K_\alpha,\pi_\alpha^\beta\colon\alpha\le\beta\le\delta\big\rangle$ is an inverse system of $\tau$-simple extensions. Then, for any $\alpha<\beta\le\delta$
and every closed set $U\sub X_\beta$ such that $|\partial U|\le\kappa$, we have $\big|\partial\pi_\alpha^\beta[U]\big|\le|\beta|\cdot\tau\cdot\kappa$.
\end{lemma}
\begin{proof}
Let us first observe that the case $\beta=\alpha+1$, where $\alpha<\delta$, follows immediately from Definition \ref{def:tau_simple_extensions}, thus we need only to prove the case where $\alpha+1<\beta$. Fix $\alpha<\delta$. The proof is by induction on $\beta>\alpha$. Let us thus fix also $\beta\le\delta$ and assume that the thesis holds for every $\alpha<\xi<\beta$. We have two cases:
\begin{enumerate}
    \item $\beta=\xi+1$ for some $\alpha<\xi<\delta$. Then, by the beginning remark, $\big|\partial\pi_\xi^\beta[U]\big|\le|\beta|\cdot\tau\cdot\kappa$. By the inductive assumption used for an ordinal number $\xi$, a closed set $\pi_\xi^\beta[U]$, and the cardinal $|\beta|\cdot\tau\cdot\kappa$, we conclude that:
\[\big|\partial\pi_\alpha^\beta[U]\big|=\big|\partial\pi^\xi_\alpha\big[\pi_\xi^\beta[U]\big]\big|\le|\xi|\cdot\tau\cdot|\beta|\cdot\tau\cdot\kappa=|\beta|\cdot\tau\cdot\kappa.\]
    \item $\beta$ is limit. First note that $w\big(K_\beta\big)\le|\beta|+\omega\le|\beta|\cdot\kappa$, because the inverse system is based on $\tau$-simple extensions. It follows that $\partial U=\bigcap_{\iota<|\beta|\cdot\kappa}A_\iota$ for some family $\big\{A_\iota\colon\ \iota<|\beta|\cdot\kappa\big\}$ of clopen subsets of $K_\beta$, which is closed under finite intersections of its elements. Then,
\[\pi^\beta_\alpha[U\setminus\partial U]=\bigcup_{\iota<|\beta|\cdot\kappa}\pi^\beta_\alpha\big[U\setminus A_\iota\big].\]
We now claim that
\[\tag{$*$}\partial\pi^\beta_\alpha[U]\sub\bigcup_{\iota<|\beta|\cdot\kappa}\partial\pi^\beta_\alpha\big[U\setminus A_\iota\big]\cup\pi^\beta_\alpha[\partial U].\]
To see this, fix $x\in\partial\pi^\beta_\alpha[U]\setminus\pi^\beta_\alpha[\partial U]$ and note that
$\big(\pi^\beta_\alpha\big)^{-1}(x)\cap U\sub U^\circ$, and hence there exists $\iota<|\beta|\cdot\kappa$ such
that $\big(\pi^\beta_\alpha\big)^{-1}(x)\cap U\sub U\setminus A_\iota$. It follows that $x\in\pi^\beta_\alpha\big[U\setminus A_\iota\big]$, and hence $x\in \partial\pi^\beta_\alpha\big[U\setminus A_\iota\big]$, because otherwise $x\in\big(\pi^\beta_\alpha\big[U\setminus A_\iota\big]\big)^\circ\sub\big(\pi^\beta_\alpha[U]\big)^\circ$, thus contradicting $x\in\partial\pi^\beta_\alpha[U]$.

Since for every $\iota<|\beta|\cdot\kappa$ the set $U\setminus A_\iota$ is clopen in $K_\beta$, for every $\iota<|\beta|\cdot\kappa$ there are $\xi_\iota\in\beta\setminus\alpha$ and clopen $B_\iota\sub K_{\xi_\iota}$ such that $U\setminus A_\iota=\big(\pi^\beta_{\xi_\iota}\big)^{-1}\big[B_\iota\big]$, and hence $\pi^\beta_\alpha\big[U\setminus A_\iota\big]=\pi^{\xi_\iota}_\alpha\big[B_\iota\big]$ for all $\iota<|\beta|\cdot\kappa$. It follows from our inductive assumption that
\[\big|\partial\pi^\beta_\alpha\big[U\setminus A_\iota\big]\big|=\big|\partial\pi^{\xi_\iota}_\alpha\big[B_\iota\big]\big|\le\big|\xi_i\big|\cdot\tau\cdot\kappa\le|\beta|\cdot\tau\cdot\kappa,\]
and hence we conclude from ($*$) that
\[\big|\partial\pi^\beta_\alpha[U]\big|\le|\beta|\cdot\kappa\cdot|\beta|\cdot\tau\cdot\kappa+\kappa\le|\beta|\cdot\tau\cdot\kappa,\]
which completes our proof.
\end{enumerate}
\end{proof}

We are in the position to prove the main theorem of this section.

\begin{theorem}\label{theorem:tau_simple_extensions_fsjnp}
Let $\tau<\frakc$ be a cardinal number. Assume that $\big\langle K_\alpha,\pi_\alpha^\beta\colon\alpha\le\beta\le\delta\big\rangle$ is an inverse system of $\tau$-simple extensions with $\delta\le\frakc$. Then, $K_\delta$ has the fsJNP.
\end{theorem}
\begin{proof}
Since $\pi^\beta_\alpha$ are irreducible for any $\alpha<\beta\leq\delta$, for every clopen $U\sub K_\beta$ we have $\big(\pi^\beta_\alpha[U]\cap\pi^\beta_\alpha\big[K_\beta\setminus U\big]\big)^\circ=\emptyset$ and hence, by Lemma \ref{lemma:image_boundary},
\[\pi^\beta_\alpha[U]\cap\pi^\beta_\alpha\big[K_\beta\setminus U\big]=\partial\pi^\beta_\alpha[U]\cup\partial\pi^\beta_\alpha\big[K_\beta\setminus U\big].\]

Let $U\sub K_\delta$ be clopen. If $\delta<\frakc$, then it follows from the above equality and Lemma \ref{lemma:size_boundaries_induction} (with $\kappa=\omega$---note that $|\partial U|=0$) that
\[\big|\pi^\delta_0[U]\cap \pi^\delta_0\big[K_\delta\setminus U\big]\big|\le|\delta|\cdot\tau\cdot\omega<\frakc.\]
If $\delta=\frakc$, then $U=\big(\pi^\delta_\beta\big)^{-1}[W]$ for some $\beta<\delta$ and clopen $W\sub K_\beta$, and hence, again by Lemma \ref{lemma:size_boundaries_induction},
\[\big|\pi^\delta_0[U]\cap\pi^\delta_0\big[K_\delta\setminus U\big]\big|=\big|\pi^\beta_0[W]\cap\pi^\beta_0\big[K_\delta\setminus W\big]\big|\le|\beta|\cdot\tau\cdot\omega<\frakc.\]
Thus, $\big|\pi^\delta_0[U]\cap\pi^\delta_0\big[K_\delta\setminus U\big]\big|<\frakc$ in any case. Since
$\pi^\delta_0[U]\cap \pi^\delta_0[X_\delta\setminus U]$ is a closed subset of $\Cantor$ of size $<\frakc$, we conclude that
it is countable, and hence it must have Lebesgue measure $0$. It remains to apply Proposition \ref{prop:transport_fsjn_seq} for $Y=K_\delta$.
\end{proof}

Rephrasing the theorem, we get that the limits of inverse systems of $\tau$-simple extensions of length at most $\frakc$ do not have the $\ell_1$-Grothendieck property, which generalizes Corollary \ref{cor:min_gen_no_ell_1_gr}.

The assumption in Theorem \ref{theorem:tau_simple_extensions_fsjnp} that the ordinal number $\delta$ is not greater than $\frakc$ seems to be essential as it allows us to appeal to Proposition \ref{prop:transport_fsjn_seq} in order to ``transport'' the fsJN-sequence $\seqn{\mu_n}$ from the Cantor space onto $K_\delta$. We do not know whether the conclusion of the theorem holds true without this assumption.

\begin{question}\label{ques:tau_simple_extensions_fsjnp_delta}
Let $\tau<\frakc$ be a cardinal number. Assume that $\big\langle K_\alpha,\pi_\alpha^\beta\colon\alpha\le\beta\le\delta\big\rangle$ is an inverse system of $\tau$-simple extensions with $\delta>\frakc$. Does $K_\delta$ necessarily have the fsJNP?
\end{question}

As an application of Theorem \ref{theorem:tau_simple_extensions_fsjnp}, we will show that some special Efimov spaces, studied by de la Vega in \cite{dlV04} and \cite{dlV05}, have the fsJNP, too.

\begin{definition}\label{def:de_la_vega_system}
A continuous inverse system $\big\langle K_\alpha,\pi_\alpha^\beta\colon\alpha\le\beta\le\omega_1\big\rangle$ of compact spaces such that $K_0=\Cantor$ and for every $\alpha<\omega_1$ the space $K_\alpha$ is homeomorphic to the space $\Cantor$ as well as there exist:
\begin{itemize}
    \item closed subsets $A_\alpha^0,A_\alpha^1\sub K_\alpha$ and a point $p_\alpha\in K_\alpha$ such that $K_\alpha=A_\alpha^0\cup A_\alpha^1$ and $A_\alpha^0\cap A_\alpha^1=\big\{p_\alpha\big\}$, and
    \item a countable group $G_\alpha$ acting on $K_\alpha$ freely, i.e. $gx\neq x$ for any $x\in K_\alpha$ and $g\in G_\alpha\setminus\big\{e_\alpha\big\}$, where $e_\alpha\in G_\alpha$ is the group identity, such that
	\[K_{\alpha+1}=\big\{(x,\phi)\in K_\alpha\times 2^{G_\alpha}\colon\ x\in gA_\alpha^{\phi(g)}\text{ for every }g\in G_\alpha\big\}\]
	and $\pi^{\alpha+1}_\alpha\big((x,\phi)\big)=x$,
\end{itemize}
is called \textit{a de la Vega system}.
\end{definition}

Assuming Jensen's $\diamondsuit$, de la Vega \cite{dlV04,dlV05} used such systems to construct hereditarily separable Efimov spaces with various homogeneity properties. Below we show that de la Vega systems are based on $\omega$-simple extensions and hence that their limits $K_{\omega_1}$ have the fsJNP.


\begin{proposition}\label{prop:delavega_omega_simple}
Every de la Vega system is based on $\omega$-simple extensions.
\end{proposition}
\begin{proof}
Let $\big\langle
K_\alpha,\pi_\alpha^\beta\colon\alpha\le\beta\le\omega_1\big\rangle$
be a de la Vega system. For each $\alpha<\omega_1$ fix
$A_\alpha^0,A_\alpha^1,p_\alpha$ and $G_\alpha$ as in the definition
of the system. We need to show that each $\pi_\alpha^\beta$ is
irreducible and that for each $\alpha<\omega_1$ and closed $U\sub
K_{\alpha+1}$ the set
$\partial\pi_\alpha^{\alpha+1}[U]\sm\pi_\alpha^{\alpha+1}[\partial
U]$ is countable.

\begin{enumerate}
    \item For every $\alpha<\beta\le\omega_1$ the function $\pi_\alpha^\beta$ is irreducible.

Since the system is continuous it is enough to show that $\pi_\alpha^{\alpha+1}$ is irreducible. To prove this we will use the following simple characterization of irreducible mappings: a function $f\colon K\to L$ between two totally disconnected compact spaces is irreducible if and only if for every clopen $U\sub K$ there is clopen $B\sub L$ such that $f^{-1}[B]\sub U$.

Thus fix $\alpha<\omega_1$ and a clopen $U\subset K_{\alpha+1}$. Without loss of generality let us assume $K_\alpha=\Cantor$. Shrinking $U$ if necessary we may assume that $\emptyset\neq U=([s]\times [t])\cap K_{\alpha+1} $ for some $s=\langle s(0),\ldots,s(n)\rangle\in 2^{n+1}$ and
$t=\langle t(g_0),\ldots,t(g_n)\rangle\in 2^{\{g_i\colon\ i\leq n\}}$, where $[s]=\big\{x\in\Cantor\colon\ x\rstr (n+1)=s\big\}$
and $[t]\subset 2^{G_\alpha}$ is defined analogously. Put:
\[W=\bigcap_{i\le n}g_i A_\alpha^{t(g_i)}\cap [s].\]
Since $\pi_\alpha^{\alpha+1}[U]=W$, $W$ is non-empty, and
$W\setminus\big\{g_ip_\alpha\colon\ i\le n\big\}$ is open in $K_\alpha$.
Fix any clopen set $B\subset W\setminus\{g_ip_\alpha\colon i\leq n\}$ and a pair
$(x,\phi)\in B\times 2^{G_\alpha}$ such that $(x,\phi)\in
K_{\alpha+1}$, i.e., $(x,\phi)\in (\pi^{\alpha+1}_\alpha)^{-1}(x)$.
Since $x\not\in\big\{g_ip_\alpha\colon i\leq n\big\}$, for every $i\le n$ there
is a unique $j_i\in 2$ such that $x\in g_i
A_\alpha^{j_i}$, and hence $\phi\big(g_i\big)=t(g_i)=j_i$ for all
$i\le n$. It follows that $(x,\phi)\in [s]\times [t]\cap
K_{\alpha+1}\subset U$, so,
summarizing,$\big(\pi^{\alpha+1}_\alpha\big)^{-1}[B]\subset U$.

    \item For every $\alpha<\omega_1$ and closed $U\sub K_{\alpha+1}$, $\Big|\partial\pi_\alpha^{\alpha+1}[U]\sm\pi_\alpha^{\alpha+1}[\partial U]\Big|\le\omega$.

We shall show that
\[\partial\pi_\alpha^{\alpha+1}[U]\sm\pi_\alpha^{\alpha+1}[\partial U]\sub\big\{gp_\alpha\colon\ g\in G_\alpha\big\}.\]
This has been almost done in the previous paragraph. Indeed,
suppose that
\[x\in\pi^{\alpha+1}_\alpha[U]\setminus\big(\pi^{\alpha+1}_\alpha[\partial U]\cup\{gp_\alpha\colon\ g\in G_\alpha\big\}\big).\]
Then, $x\in\pi^{\alpha+1}_\alpha[U^\circ]\setminus\big\{gp_\alpha\colon\ g\in G_\alpha\big\}$, and, by the same argument as
in (1), we get a clopen set $B\subset K_\alpha$ containing $x$ and such that $\big(\pi^{\alpha+1}_\alpha\big)^{-1}(x)\sub\big(\pi^{\alpha+1}_\alpha\big)^{-1}[B]\sub U^\circ$, and therefore $x\in B\sub\big(\pi^{\alpha+1}_\alpha\big)[U]^\circ$. But this implies that $x\not\in\partial\pi^{\alpha+1}_\alpha[U]$, which completes the proof.
\end{enumerate}
\end{proof}

Recall that a space $X$ is called \textit{rigid} if it has no non-trivial autohomeomorphisms, i.e. if $f\colon X\to X$ is a homeomorphism, then $f$ is the identity map. Combining Theorem \ref{theorem:tau_simple_extensions_fsjnp} with Proposition \ref{prop:delavega_omega_simple} and de la Vega's \cite[Theorems 5.1 and 5.2]{dlV04}, we get the following corollary, important in the view of  \cite[Example 15]{KS18}.

\cordelavegafsjnp

Furthermore, using the generalizations of \cite{dlV05} obtained by
 Back\'e \cite{Back18}, we get the next corollary.

\corbackefsjnp

%

Let us not here that we do not know whether the classes of compact spaces obtained by simple extensions and $\tau$-simple extensions for $\tau\in[\omega,\frakc]$ are essentially different. Thus, we ask the following questions.

\begin{question}\label{ques:tau_simple_ext_simple_ext}\ 
\begin{enumerate}
    \item Does there exist a compact space which is the limit of an inverse system based on $\omega$-simple extensions but is not the limit of any inverse system based on simple extensions?
    \item Assume that $\frakc>\omega_1$. Does there exist a compact space which is the limit of an inverse system based on $\omega_1$-simple extensions but is not the limit of any inverse system based on simple extensions?
\end{enumerate}
\end{question}

The following problem is a special case of Question \ref{ques:tau_simple_ext_simple_ext}.

\begin{question}\label{ques:dlV_simple}
Does there (consistently) exist a de la Vega system whose limit cannot be represented as the limit of an inverse system based on simple extensions?
\end{question}

\subsection*{Acknowledgements}
We would like to thank Grzegorz Plebanek for providing to us many valuable ideas and comments which helped us to obtain results presented in this paper.

The research of the first  named author was supported by the GA\v{C}R project 20-22230L and
RVO: 67985840. The second and third named authors were supported by the Austrian Science Fund FWF, Grants I 2374-N35, I 3709-N35, M 2500-N35.


\end{document}